\documentclass[a4paper,11pt]{article}
\usepackage{graphicx}
\usepackage[english]{babel}
\usepackage[utf8x]{inputenc}
\usepackage{fullpage}
\usepackage{amsmath}
\usepackage{amssymb}
\usepackage{theorem}
\usepackage{stmaryrd}

\newtheorem{theorem}{Theorem}[section]
\newtheorem{lemma}[theorem]{Lemma}
\newtheorem{proposition}[theorem]{Proposition}
\newtheorem{corollary}[theorem]{Corollary}

\theorembodyfont{\rm} 
\newtheorem{definition}[theorem]{Definition}
\newtheorem{remark}[theorem]{Remark}

\newenvironment{proof}{\noindent {\it Proof}.}{\hfill$\Box$}

\newcommand{\lignelarge}{\lower1.5ex\hbox{\rule{0ex}{4ex}}}
\newcommand{\Int}[1]{{\rm Int}\left(#1\right)}
\newcommand{\IN}{{\mathbb N}}
\newcommand{\CA}{{\mathcal A}}
\newcommand{\CC}{{\mathcal C}}
\newcommand{\CF}{{\mathcal F}}
\newcommand{\CG}{{\mathcal G}}

\newcommand{\milieu}{{\rm mid}}

\let\Lbrack\llbracket
\let\Rbrack\rrbracket

\begin{document}

\title{Interval maps of given topological entropy and Sharkovskii's type}
\author{Sylvie Ruette}
\date{June 9, 2019}
\maketitle

\begin{abstract}
It is known that the topological entropy of a continuous interval map
$f$ is positive if and only if the type of $f$  for Sharkovskii's order
is $2^d p$ for some odd integer $p\ge 3$ and some $d\ge 0$; and
in this case the topological entropy of $f$ is greater than or equal to
$\frac{\log\lambda_p}{2^d}$, where $\lambda_p$ is the unique positive
root of $X^p-2X^{p-2}-1$. 
For every odd $p\ge 3$, every $d\ge 0$
and every $\lambda\ge\lambda_p$, we build a piecewise monotone
continuous interval map that is of type $2^dp$ for Sharkovskii's order
and whose topological entropy is $\frac{\log\lambda}{2^d}$. This shows that,
for a given type, every possible finite entropy above the minimum
can be reached provided the type allows the map to have positive entropy.
Moreover, if $d=0$ the map we build is topologically mixing.
\end{abstract}

\section{Introduction}

In this paper, an interval map is a continuous map $f\colon I\to I$ where
$I$ is a compact nondegenerate interval. A point $x\in I$ is periodic
of period $n$
if $f^n(x)=x$ and $n$ is the least positive integer with this property, i.e.
$f^k(x)\ne x$ for all $k\in\Lbrack 1,n-1\Rbrack$.

Let us recall Sharkovskii's theorem and the definitions of Sharkovskii's order
and type \cite{Sha} (see e.g. \cite[Section 3.3]{R3}).

\begin{definition}\label{defi:sharkovskii}
\emph{Sharkovskii's order} is the total ordering on $\IN$ defined by:
$$
3\lhd 5\lhd 7\lhd 9\lhd\cdots\lhd 2\cdot 3 \lhd 2\cdot 5\lhd\cdots\lhd
2^2\cdot 3\lhd 2^2\cdot 5\lhd\cdots\lhd 2^3\lhd 2^2 \lhd 2 \lhd 1
$$
(first, all odd integers $n>1$, then $2$ times the odd integers $n>1$,
then successively $2^2\times$, $2^3\times$, \ldots, $2^k\times$ $\ldots$ the odd integers $n>1$, 
and finally
all the powers of $2$ by decreasing order).

$a \rhd b$ means $b\lhd a$. The notation $\unlhd, \unrhd$ will denote 
the order with possible equality.
\end{definition}

\begin{theorem}[Sharkovskii]\label{theo:Sharkovsky}
If an interval map $f$ has a periodic point of period
$n$ then, for all integers  $m\rhd n$, $f$ has periodic points of period $m$.
\end{theorem}

\begin{definition}
Let $n\in\IN\cup\{2^\infty\}$.
An interval map $f$ is \emph{of type $n$ (for
Sharkovskii's order)} if the periods of the periodic points of $f$
form exactly the set $\{m\in\IN\mid m\unrhd n\}$, where the notation
$\{m\in\IN\mid m\unrhd 2^\infty\}$ stands for $\{2^k\mid k\ge 0\}$.
\end{definition}

It is well known that an interval map $f$ has positive topological entropy 
if and
only if its type is $2^d p$ for some odd integer $p\ge 3$ and some $d\ge 0$
(see e.g. \cite[Theorem~4.58]{R3}).
The entropy of such a map is bounded from below 
(see theorem~4.57 in \cite{R3}).

\begin{theorem}[Štefan, Block-Guckenheimer-Misiurewicz-Young]\label{theo:lambdap}
Let $f$ be an interval map of type $2^d p$ for some odd integer $p\ge 3$ 
and some $d\ge 0$. Let $\lambda_p$ be the
unique positive root of $X^p-2X^{p-2}-1$. Then $\lambda_p>\sqrt{2}$ and
$h_{top}(f)\ge \frac{\log \lambda_p}{2^d}$.
\end{theorem}

This bound is sharp: for every
$p,d$, there exists a interval map of type $2^dp$ and topological entropy
$\frac{\log \lambda_p}{2^d}$. These examples were first introduced by
Štefan, although the entropy of these maps was computed later  
\cite{Ste, BGMY}.

Moreover, it is known that the type of a topological mixing interval map
is $p$ for some odd integer $p\ge 3$ (see e.g. \cite[Proposition~3.36]{R3}). 
The Štefan maps of type $p$ are topologically mixing 
\cite[Example~3.21]{R3}.

\medskip
We want to show that the topological entropy of a piecewise monotone
map can be equal to any real number, the lower bound of
Theorem~\ref{theo:lambdap} being the only restriction. First, 
for every odd integer $p\ge 3$ and every real number $\lambda\ge \lambda_p$,
we are going to build a piecewise monotone map
$f_{p,\lambda}\colon [0,1]\to [0,1]$ such that its type is $p$ for
Sharkovskii's order, its topological entropy is $\log\lambda$,
and the map is topologically mixing.
Then we will show that for every odd integer $p\ge 3$, every integer $d\ge 0$
and every real number $\lambda\ge \lambda_p$, there exists a
piecewise monotone interval map $f$ such that its type is $2^d p$ for
Sharkovskii's order and its topological entropy is $\frac{\log\lambda}{2^d}$.

\subsection{Notations}

We say that an interval is \emph{degenerate} if it is either empty or reduced
to one point, and \emph{nondegenerate} otherwise.  When we consider an interval
map $f\colon I\to I$, every interval is implicitly a subinterval of
$I$.

Let $J$ be a nonempty interval. Then $\partial J:=\{\inf J, \sup J\}$
are the endpoints of $J$ (they may be equal if $J$ is reduced to one
point) and $|J|$ denotes the length of $J$ (i.e. $|J|:=\sup J-\inf J$).
Let $\milieu(J)$ denote the middle point of 
$J$, that is, $\milieu(J):=\frac{\inf J+\sup J}2$.

An interval map $f\colon I\to I$ is  \emph{piecewise monotone} 
if there exists a finite partition of $I$ into intervals such that $f$ is
monotone on each element of this partition.

An interval map $f$ has a \emph{constant slope} $\lambda$
if $f$ is piecewise monotone and if on each of its pieces of monotonicity
$f$ is linear and the absolute value of the slope coefficient is $\lambda$.

\section{Štefan maps}

We recall the definition of the Štefan maps of odd type $p\ge 3$.

\medskip
Let $n\ge 1$ and $p:=2n+1$.
The Štefan map $f_p\colon [0,2n]\to [0,2n]$, 
represented in Figure~\ref{fig:type-2n+1},
is defined as follows: it
is linear on $[0,n-1]$, $[n-1,n]$, $[n,2n-1]$ and $[2n-1,2n]$, and
$$
f_p(0):=2n,\ f_p(n-1):=n+1,\ f_p(n):=n-1,\ f_p(2n-1):=0,\ f_p(2n):=n.
$$
Note that $n=1$ is a particular case because $0=n-1$ and $n=2n-1$.
\begin{figure}[htb]
\centerline{\includegraphics{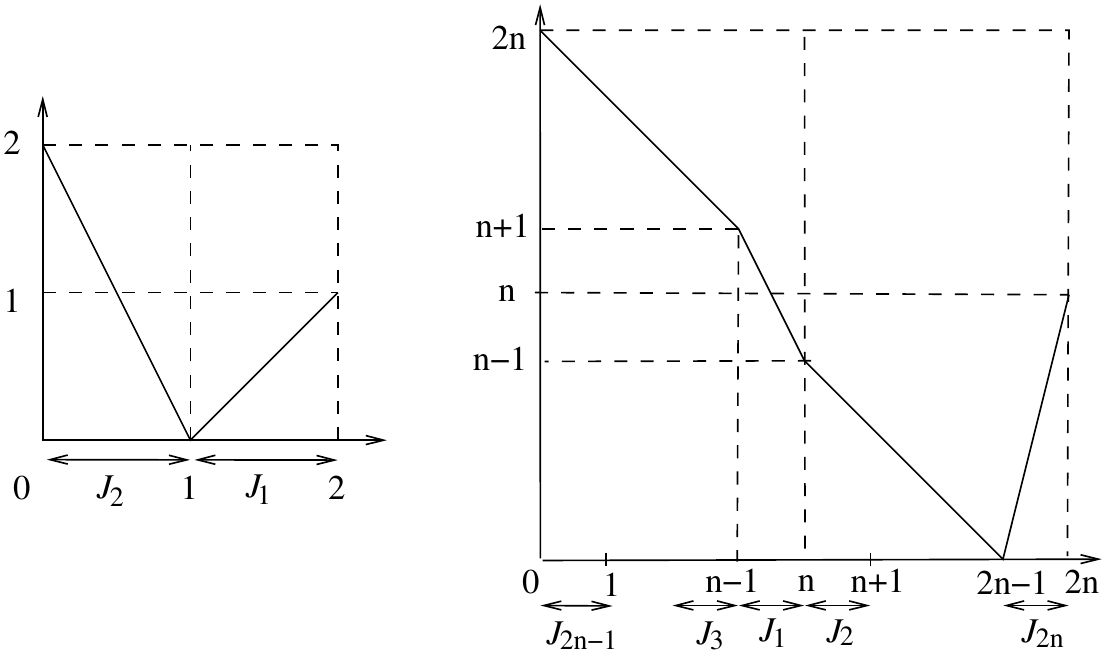}}
\caption{On the left:  the map $f_3$. On the right: 
the map $f_p$ with $p=2n+1>3$.}
\label{fig:type-2n+1}
\end{figure}

Next proposition summarises the properties of $f_p$, see
\cite[Example~3.21]{R3} for the proof.

\begin{proposition}
The map $f_p$ is topologically mixing, its type for Sharkovskii's order
is $p$ and $h_{top}(f)=\log\lambda_p$. Moreover, the point $n$ is
periodic of period $p$, and $f_p^{2k-1}(n)=n-k$ and $f_p^{2k}(n)=n+k$
for all $k\in\Lbrack 1,n\Rbrack$.
\end{proposition}

\section{Mixing map of given entropy and odd type}\label{sec:oddtype}

For every odd integer $p\ge 3$ and every real number $\lambda\ge \lambda_p$,
we are going to build a piecewise monotone continuous map
$f_{p,\lambda}\colon [0,1]\to [0,1]$ such that its type is $p$ for
Sharkovskii's order, its topological entropy is $\log\lambda$,
and the map is topologically mixing. We will write $f$ instead
of $f_{p,\lambda}$ when there is no ambiguity on $p,\lambda$.

The idea is the following: we start with the Štefan map $f_p$,
we blow up the minimum into an interval and we define the map of this
interval in such a way that the added dynamics increases the
entropy without changing the type. At the same time, we make the slope
constant and equal to $\lambda$, so that the entropy is $\log \lambda$ 
according to the following theorem
\cite[Corollary 4.3.13]{ALM}, which is due to Misiurewicz-Szlenk
\cite{MS2}, Young \cite{You} and Milnor-Thurston \cite{MT2}.

\begin{theorem}\label{theo:entropie-penteCte}
Let $f$ be a piecewise monotone interval map.
Suppose that $f$ has a constant slope $\lambda\ge 1$. 
Then $h_{top}(f)=\log\lambda$.
\end{theorem}

We will also need the next result
(see the proof of Lemma~4.56 in \cite{R3}).

\begin{lemma}\label{lem:lambdap}
Let $p\ge 3$ be an odd integer and $P_p(X):=X^p-2X^{p-2}-1$.
Then $P_p(X)$ has a unique positive root, denoted by
$\lambda_p$. Moreover,
$P_p(x)<0$ if $x\in [0,\lambda_p[$ and $P_p(x)>0$ if $x> \lambda_p$.

Let $\chi_p(X):=X^{p-1}-X^{p-2}-\sum_{i=0}^{p-2}(-X)^i$. 
Then $P_p(X)=(X+1)\chi_p(X)$, and thus
$\chi_p(x)<0$ if $x\in [0,\lambda_p[$ and $\chi_p(x)>0$ if $x> \lambda_p$.
\end{lemma}

\subsection{Definition of the map}

We fix an odd integer $p\ge 3$ and a real $\lambda\ge \lambda_p$.
Recall that $\lambda_p>\sqrt2>1$ (Theorem~\ref{theo:lambdap}).

\medskip
We are going to define points ordered as follows:
\begin{gather*}
x_{p-2}<x_{p-4}<\cdots <x_1<x_0<x_2<x_4<\cdots<x_{p-3}
\le t<x_{p-1},\\
\text{with}\quad x_{p-2}=0,\ x_{p-3}=\frac1\lambda\quad\text{and}
\quad x_{p-1}=1.
\end{gather*}
The points $x_0,\ldots, x_{p-1}$ will form a periodic orbit of period $p$,
that is, $f(x_i)=x_{i+1\bmod p}$ for all $i\in\Lbrack 0, p-1\Rbrack$.

\begin{remark}
In the following construction, the case $p=3$ is degenerate. The periodic
orbit is reduced to $x_1<x_0<x_2$ with $x_1=0, 
x_0=\frac1\lambda, x_2=1$. We only have to determine the value of $t$.
\end{remark}

The map $f\colon [0,1]\to [0,1]$ is defined as follows
(see Figure~\ref{fig:f}):
\begin{itemize}
\item $f(x):=1-\lambda x$ for all $x\in[0,\frac{1}{\lambda}]=[x_{p-2},x_{p-3}]$
(so that $f(0)=1$ and $f(\frac{1}{\lambda})=0$),
\item $f(x):=\lambda(x-t)$ for all $x\in [t,1]$ (so that $f(t)=0$
and $f(1)=\lambda(1-t)$),
\item definition on $[\frac{1}{\lambda},t]$: 
we want to have $f([\frac{1}{\lambda},t])\subset [0,x_{p-4}]$
(note that $x_{p-4}$ is the least positive point among $x_0,\ldots,x_{p-1}$),
with $f$ of constant slope $\lambda$, in such a way that all the critical points
except at most one are sent by $f$ on either $0$ or $x_{p-4}$.
If $t=\frac{1}{\lambda}$, there is nothing to do. If $t>\frac{1}{\lambda}$,
we set
$\ell:= t-\frac{1}{\lambda}$ (length of the interval), 
$k:=\left\lfloor \frac{\lambda \ell}{2x_{p-4}}\right\rfloor$,
\begin{equation}\label{eq:defJi}
J_i:=\left[\frac{1}{\lambda}+(i-1) \frac{2x_{p-4}}{\lambda},
\frac{1}{\lambda}+ i \frac{2x_{p-4}}{\lambda}\right]
\quad\text{for all }i\in\Lbrack 1, k\Rbrack,
\end{equation}
\begin{equation}\label{eq:defK}
K:=\left[\frac{1}{\lambda}+ k\frac{2x_{p-4}}{\lambda}, t\right].
\end{equation}
If $p=3$, we replace $x_{p-4}$ (not defined) by $1$ in the above definitions.

It is possible that there is no interval $J_1,\ldots, J_k$
(if $k=0$) or that $K$ is reduced to the point $\{t\}$.
On each interval $J_1,\ldots J_k$, $f$ is defined
as the tent map of summit $x_{p-4}$: $f(\min J_i)=0$,
$f$ is increasing of slope $\lambda$ on $[\min J_i,\milieu(J_i)]$
(thus $f(\milieu(J_i))=x_{p-4}$ because of the length of $J_i$),
then $f$ is decreasing of slope $-\lambda$ on $[\milieu(J_i),\max J_i]$
and $f(\max J_i)=0$.
On $K$, $f$ is defined as a tent map with a summit $< x_{p-4}$: $f(\min K)=0$, 
$f$ is increasing of slope $\lambda$ on $[\min K,\milieu(K)]$,
then $f$ is decreasing of slope $-\lambda$ on $[\milieu(K),\max K]$
and $f(\max K)=0$.
\end{itemize}

\begin{figure}[htb]
\centerline{\includegraphics[width=15cm]{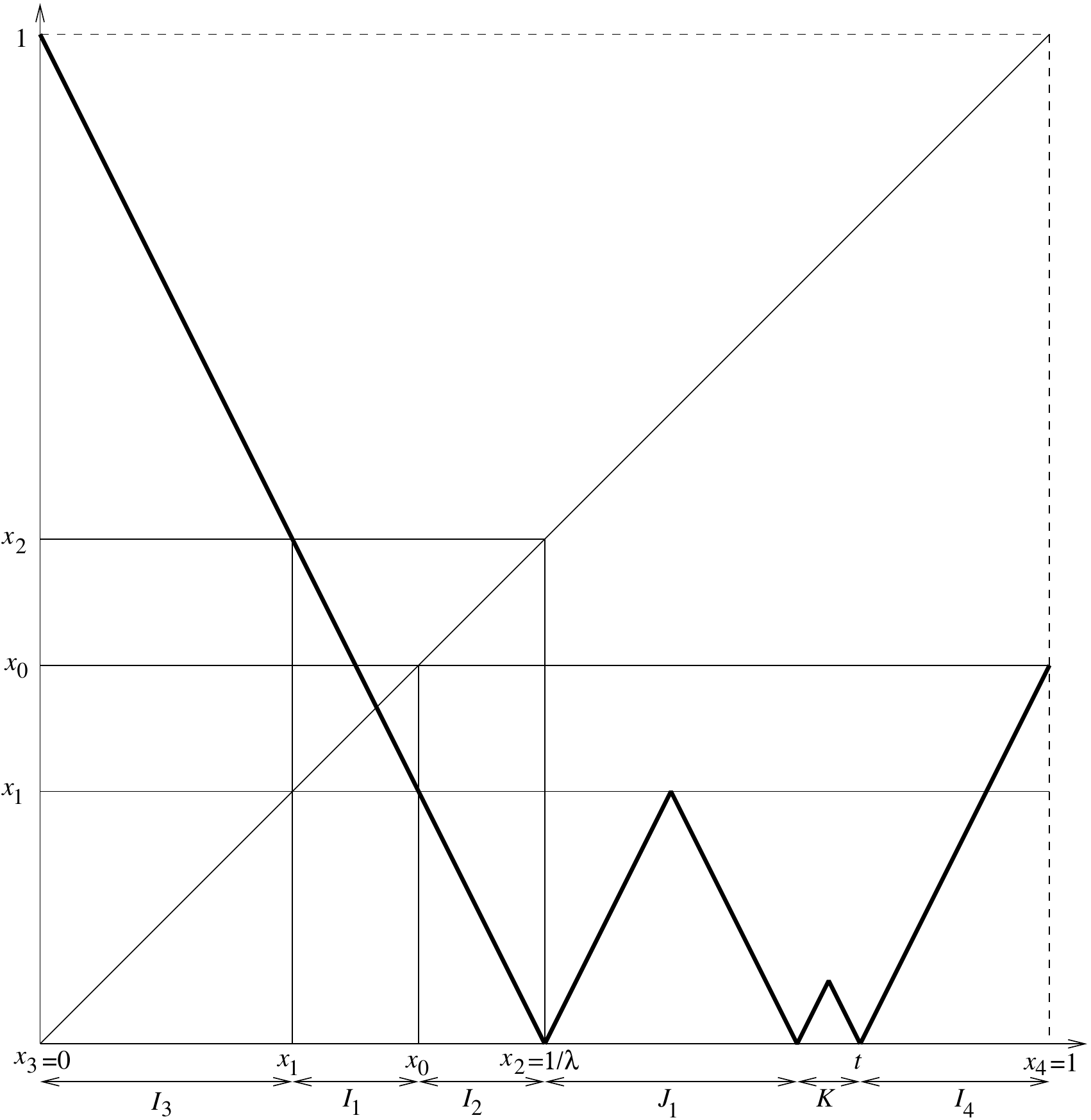}}
\caption{The map $f$ for $p=5$ and $\lambda=2$. \label{fig:f}}
\end{figure}

In this way, we get a map $f$ that is continuous on $[0,1]$, piecewise
monotone, of constant slope $\lambda$. It remains to define $t$ and the
points $\{x_i\}_{0\le i\le p-4}$ (recall that
$x_{p-3}=\frac{1}{\lambda}$, $x_{p-2}=0$ and $x_{p-1}=1$).

We want these points to satisfy:
\begin{equation}
x_0\ =\ \lambda(1-t)\label{eq:t}
\end{equation}
and
\begin{equation}\label{eq:xi}
\left\{\begin{array}{lcl}
x_1&=&1-\lambda x_0\\
x_2&=&1-\lambda x_1\\
&\vdots&\\
x_{p-3}&=&1-\lambda x_{p-4}
\end{array}\right.
\end{equation}
and to be ordered as follows:
\begin{gather}
x_{p-2}<x_{p-4}<\cdots <x_1<x_0<x_2<x_4<\cdots<x_{p-3}
\label{eq:ordre-xi}\\
\frac{1}{\lambda} \le t<x_{p-1}.\label{eq:ordre-t}
\end{gather}
If $p=3$, the system \eqref{eq:xi} is empty, and equation~\eqref{eq:ordre-xi}
is satisfied because it reduces to $0=x_1<x_0=\frac{1}{\lambda}$.

According to the definition of $f$, the equations \eqref{eq:t}, \eqref{eq:xi},
\eqref{eq:ordre-xi},
\eqref{eq:ordre-t} imply that $f(x_i)=x_{i+1}$ for all $i\in\Lbrack 0,
p-2\Rbrack$ and $f(x_{p-1})=x_0$.

We are going to show that the system \eqref{eq:xi} is equivalent to:
\begin{equation}\label{eq:formule-xi}
\forall i\in\Lbrack 0,p-4\Rbrack,\quad x_i=\frac{(-1)^i}{\lambda^{p-i-2}}
\sum_{j=0}^{p-i-3}(-\lambda)^j.
\end{equation}
We use a descending induction on $i$.

$\bullet$ According to the last line of \eqref{eq:xi}, 
$x_{p-4}=\frac{1}{\lambda}(1-x_{p-3})=\frac{1}{\lambda^2}(\lambda-1)$.
This is \eqref{eq:formule-xi} for $i=p-4$.

$\bullet$ Suppose that \eqref{eq:formule-xi} holds for $i$ with $i\in
\Lbrack 1,p-4\Rbrack$. By \eqref{eq:xi}, $x_i=1-\lambda x_{i-1}$, thus
\begin{align*}
x_{i-1}&=-\frac{1}{\lambda}(x_i-1)\\
&=-\frac{(-1)^i}{\lambda^{p-i-1}}\left(\sum_{j=0}^{p-i-3}(-\lambda)^j-
(-1)^i\lambda^{p-i-2}\right)
\end{align*}
Since $p$ is odd, $-(-1)^i\lambda^{p-i-2}=(-\lambda)^{p-i-2}$. Hence
\[
x_{i-1}=\frac{(-1)^{i-1}}{\lambda^{p-i-1}}\sum_{j=0}^{p-i-2}(-\lambda)^j,
\]
which gives \eqref{eq:formule-xi} for $i-1$. This ends the proof of
\eqref{eq:formule-xi}.

Equation~\eqref{eq:t} is equivalent to
$t=1-\frac{1}{\lambda}x_0$.
Thus, using \eqref{eq:formule-xi}, we get
\begin{equation}\label{eq:formule-t}
t=\frac{1}{\lambda^{p-1}}\left(\lambda^{p-1}-\sum_{j=0}^{p-3}(-\lambda)^j
\right).
\end{equation}

Conclusion: with the values of $x_0,\ldots, x_{p-4}$ and $t$ given by
\eqref{eq:formule-xi} and \eqref{eq:formule-t}, the system of equations
\eqref{eq:t}-\eqref{eq:xi} is satisfied (and there is a unique solution).
It remains to show that these points are ordered as stated in
\eqref{eq:ordre-xi} and \eqref{eq:ordre-t}.

Let $i$ be in $\Lbrack 0, p-6\Rbrack$. By \eqref{eq:formule-xi}, we have
\begin{align*}
x_{i+2}-x_i&=\frac{(-1)^i}{\lambda^{p-i-2}}\left(
\lambda^2\sum_{j=0}^{p-i-5}(-\lambda)^j-\sum_{j=0}^{p-i-3}(-\lambda)^j\right)\\
&=\frac{(-1)^i}{\lambda^{p-i-2}}\left(
\sum_{j=2}^{p-i-3}(-\lambda)^j-\sum_{j=0}^{p-i-3}(-\lambda)^j\right)\\
&=\frac{(-1)^i}{\lambda^{p-i-2}}(\lambda-1)
\end{align*}
Since $\lambda-1>0$, we have, for all $i\in \Lbrack 0, p-6\Rbrack$,
\begin{itemize}
\item $x_i<x_{i+2}$ if $i$ is even,
\item $x_{i+2}<x_i$ if $i$ is odd.
\end{itemize}
By \eqref{eq:formule-xi}, $x_{p-4}=\frac{\lambda-1}{\lambda^2}$. Since
$\lambda>1$, $x_{p-4}>0=x_{p-2}$.
Again by \eqref{eq:formule-xi},
\begin{align*}
x_0-x_1&=\frac{1}{\lambda^{p-2}}\left(
\sum_{j=0}^{p-3}(-\lambda)^j+\lambda\sum_{j=0}^{p-4}(-\lambda)^j\right)\\
&=\frac{1}{\lambda^{p-2}}\left(
\sum_{j=0}^{p-3}(-\lambda)^j-\sum_{j=1}^{p-3}(-\lambda)^j\right)\\
&=\frac{1}{\lambda^{p-2}}>0
\end{align*}
thus $x_0<x_1$.
Moreover,
\[
x_{p-3}-x_{p-5}=\frac{1}{\lambda}-\frac{\lambda^2-\lambda+1}{\lambda^3}
=\frac{\lambda-1}{\lambda^3}>0
\]
thus $x_{p-5}<x_{p-3}$.
This several inequalities imply \eqref{eq:ordre-xi}.

By \eqref{eq:formule-t}, we have
\[
t-\frac{1}{\lambda}=\frac{1}{\lambda^{p-1}}\left(
\lambda^{p-1}-\lambda^{p-2}-\sum_{j=0}^{p-3}(-\lambda)^j\right)=
\frac{1}{\lambda^{p-1}}\cdot \chi_p(\lambda),
\]
where $\chi_p$ is defined in Lemma~\ref{lem:lambdap}. 
According to this lemma,
$\chi_p(\lambda)\ge 0$ (with equality iff $\lambda=
\lambda_p$) because $\lambda
\ge \lambda_p$. This implies that $t\ge \frac{1}{\lambda}$ 
(with equality iff $\lambda=\lambda_p$).
Moreover, if $t\ge 1$, then $x_0=\lambda(1-t)\le 0$, 
which is impossible by \eqref{eq:ordre-xi}; thus $t<1$. 
Therefore, the inequalities~\eqref{eq:ordre-t} hold.

Finally, we have shown that the map $f_{p,\lambda}=f$ is defined as wanted.

\subsection{Entropy}

\begin{corollary}
$h_{top}(f_{p,\lambda})=\log\lambda$.
\end{corollary}

\begin{proof}
This result is given by Theorem~\ref{theo:entropie-penteCte} because,
by definition, $f_{p,\lambda}$ is piecewise monotone of constant slope 
$\lambda$  with $\lambda>1$. 
\end{proof}

\subsection{Type}

\begin{lemma}\label{lem:periode-pseudo-recouvrement}
Let $g\colon [0,1]\to [0,1]$ be a continuous map.
Let $\CA$ be a finite family of closed intervals that form a pseudo-partition
of $[0,1]$, that is,
\[
\bigcup_{A\in\CA}A=[0,1]\quad\text{and}\quad\forall A,B\in\CA,\ A\neq B
\Rightarrow \Int{A}\cap \Int{B}=\emptyset.
\]
We set
$
\partial\CA=\bigcup_{A\in\CA}\partial A.
$
Let $\CG$ be the oriented graph whose vertices are the elements of $\CA$
and in which there is an arrow $A\dashrightarrow B$ iff
$g(A)\cap \Int{B}\neq \emptyset$.
Let $x$ be a periodic point of period $q$ for $g$ such that
$\{g^n(x)\mid n\ge 0\}\cap \partial \CA=\emptyset$. Then there
exist $A_0,\ldots, A_{q-1}\in\CA$ such that
$A_0\dashrightarrow A_1\dashrightarrow\cdots\dashrightarrow A_{q-1}
\dashrightarrow A_0$ is a cycle in the graph $\CG$.
\end{lemma}

\begin{proof}
For every $n\ge 0$, there exists a unique element $A_n\in\CG$ such that
$g^n(x)\in \Int{A_n}$ because
$\{g^n(x)\mid n\ge 0\}\cap \partial \CA=\emptyset$.
We have $g^n(x)\in A_n$ and $g^{n+1}(x)\in\Int{A_{n+1}}$, thus
$g(A_n)\cap \Int{A_{n+1}}\neq\emptyset$; in other words, there is an
arrow $A_n\dashrightarrow A_{n+1}$ in $\CG$. Finally, $A_q=A_0$ because
$g^q(x)=x$.
\end{proof}

\begin{proposition}
The map $f_{p,\lambda}$ is of type $p$ for Sharkovskii's order.
\end{proposition}

\begin{proof}
According to the definition of $f=f_{p,\lambda}$, 
$x_0$ is a periodic point of period $p$. It remains to
show that $f$ has no periodic point of period $q$ with
$q$ odd and $3\le q<p$.

We set $I_1:=\langle x_0,x_1\rangle$, $I_i:=\langle x_{i-2},x_i\rangle$
for all $i\in\Lbrack 2,p-2\Rbrack$ and
$I_{p-1}:=[t,1]$, where $\langle a,b\rangle$
denotes the convex hull of $\{a,b\}$ (i.e. $\langle a,b\rangle=[a,b]$ 
or $[b,a]$). The intervals $J_i, K$ have been defined in \eqref{eq:defJi} and
\eqref{eq:defK}. 
The family $\CA:=\{I_1,\ldots, I_{p-1},J_1,\ldots, J_k, K\}$ is a
pseudo-partition of $[0,1]$. Let $\CG$ be the oriented graph associated
to $\CA$ for the map $g=f$ as defined in 
Lemma~\ref{lem:periode-pseudo-recouvrement}.
If $f(A)\supset B$, the arrow 
$A \dashrightarrow B$ is replaced by $A\to B$  (full covering).
The graph $\CG$ is represented in
Figure~\ref{fig:Grecouvrement}; a dotted arrow
$A \dashrightarrow B$ means that $f(A)\cap \Int{B}\neq \emptyset$ but
$f(A)\not\supset B$ (partial covering).

\begin{figure}[htb]
\centerline{\includegraphics{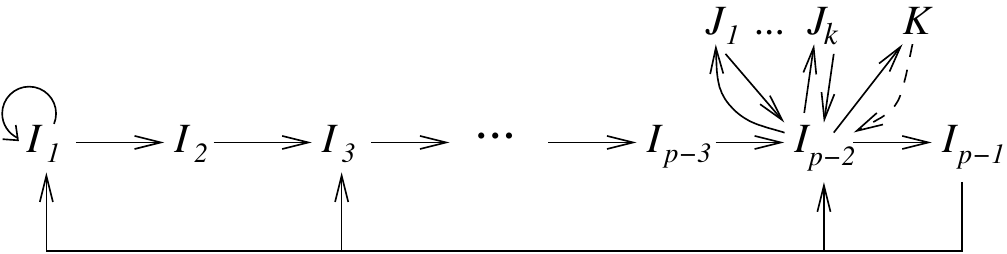}}
\caption{Covering graph $\CG$ associated to $f$. \label{fig:Grecouvrement}
}
\end{figure}

The subgraph associated to the intervals $I_1,\ldots, I_{p-1}$
is the graph associated to a Štefan cycle of period $p$  (see
\cite[Lemma~3.16]{R3}). The only additional arrows with respect to
the Štefan graph are between the intervals $J_1,\ldots, J_k, K$ on the one
hand and $I_{p-2}$ on the other hand. There is only one partial covering,
which is $K\dashrightarrow I_{p-2}$. 

Let $q$ be an odd integer with $3\le q<p$. We easily see that,
in this graph, there is no primitive cycle of length $q$
(a cycle is primitive if it is not the repetition of a shorter cycle):
the cycles not passing through $I_1$ have an even length, whereas the cycles
passing through $I_1$ have a length either equal to $1$, or greater than
or equal to $p-1$. Moreover, if $x$ is a periodic point of period $q$,
then $\{f^n(x)\mid n\ge 0\}\cap \partial \CA=\emptyset$
(because the periodic points in $\partial\CA$ are of period $p$).
According to Lemma~\ref{lem:periode-pseudo-recouvrement}, 
$f$ has no periodic point of period $q$.
Conclusion: $f$ is of type $p$ for Sharkovskii's order.
\end{proof}

\subsection{Mixing}

\begin{proposition}\label{prop:melange}
The map $f_{p,\lambda}$ is topologically mixing.
\end{proposition}

\begin{proof}
This proof is inspired by \cite[Lemmas 2.10, 2.11]{R3} and their use in 
\cite[Example 2.13]{R3}.

We will use several times that the image by $f=f_{p,\lambda}$ 
of a nondegenerate interval
is a nondegenerate interval (and thus all its iterates are nondegenerate).

Let $A$ be a nondegenerate closed interval included in $[0,1]$.
We are going to show that there exists an integer $n\ge 0$ 
such that $f^n(A)=[0,1]$.

We set
\[
\CC_0:=\bigcup_{i=1}^k\partial J_i \cup \{t\},\quad
\CC_1:=\{\milieu(J_i)\mid i\in\Lbrack 1,k\Rbrack\},\quad
c_K:=\milieu(K).
\]
The set of critical points of $f$ is
$\CC_0\cup \CC_1\cup \{c_K\}$.

\medskip
\textbf{Step 1:} there exists $i_0\ge 0$ such that 
$f^{i_0}(A)\cap (\CC_0\cup\CC_1)\ne\emptyset$ and 
there exists $n_0\ge 0$ such that $0\in f^{n_0}(A)$.

Let 
\begin{gather*}
J_i':=[\min J_i,\milieu(J_i)]\text{ and } J_i'':=[\milieu(J_i),\max J_i]
\quad\text{for all }i\in\Lbrack 1,k\Rbrack,\\
\CF:=\left\{\left[0,\frac{1}{\lambda}\right], [t,1], K\right\}\cup\{J_i', J_i''\mid 
i\in\Lbrack 1,k\Rbrack\}.
\end{gather*}

If $A\subset B$ for some  $B\in \CF$ and $B\ne K$, then $|f(A)|=\lambda |A|$.
If $A\subset K$, then $|f(A)|\ge \frac{\lambda |A|}2$
and $f(A)\subset I_{p-2}$, thus $|f^2(A)|=\lambda |f(A)|\ge 
\frac{\lambda^2}2 |A|$.
We have $\lambda>1$ and $\frac{\lambda^2}2>1$ because $\lambda>\sqrt{2}$ 
(Theorem~\ref{theo:lambdap}). If for all $i\ge 0$,
there exists $A_i\in\CF$ such that $f^i(A)\subset A_i$, then what precedes
implies that
$\lim_{i\to+\infty}|f^i(A)|=+\infty$. This is impossible because
$f^i(A)\subset [0,1]$. Thus there exist $i_0\ge 0$ and $c\in \CC_0\cup \CC_1$
such that $c\in f^{i_0}(A)$.
If $c\in \CC_0$, then $f(c)=0$, and hence $0\in f^{i_0+1}(A)$. If
$c\in \CC_1$, then $f(c)=x_{p-4}$ and hence $0\in f^{i_0+3}(A)$.
This ends step 1.

\medskip
\textbf{Step 2:} there exist $n_1\ge n_0$ and $j\in\Lbrack 1, p-1\Rbrack$
such that $f^{n_1}(A)\supset I_j$.

Recall that $I_1=[x_1,x_0]$, $I_i=\langle x_{i-2},x_i\rangle$
for all $2\le i\le p-2$ and $I_{p-1}=[t,1]=[t,x_{p-1}]$. We set $I_0:=I_1$. 
By definition, for all $0\le i\le p-1$, there exists $\delta_i>0$ such that
$I_i=\langle x_i, x_i+(-1)^{i+1}\delta_i\rangle$. Moreover, $f$ is linear 
of slope $-\lambda$ on each of the intervals $I_0,\ldots, I_{p-2}$
and of slope $+\lambda$ on $I_{p-1}$.

We set $B_{-2}:=f^{n_0}(A)$. This is a nondegenerate closed interval
containing $0$, thus there exists $b>0$ such that $B_{-2}=[0,b]$
with $0=x_{p-2}$. We set $B_i:=f^{i+2}(B_{-2})$ for all $i\ge -2$, 
and we define $m\ge -2$ as the least integer such that $B_m$ is not included
in a interval of the form $I_j$ (such an integer $m$ exists by step 1).

If $b>x_{p-4}$, then $B_{-2}\supset I_{p-2}$ and $m=-2$. Otherwise,
$B_{-2}\subset I_{p-2}$ and $B_{-1}=[1-\lambda b, 1]
=[x_{p-1}-\lambda b, x_{p-1}]$ because $f|_{I_{p-2}}$
is of slope $-\lambda$. If $1-\lambda b<t$, then $B_{-1}\supset I_{p-1}$
and $m=-1$. Otherwise, $B_{-1}\subset I_{p-1}$ and $B_0=[x_0-\lambda^2 b, x_0]$
because $f|_{I_{p-1}}$ is of slope $+\lambda$. We go on in a similar way.
\begin{itemize}
\item If $m>0$, then $B_0\subset I_0$ and $B_1=[x_1,x_1+\lambda^3 b]$.\\
\item If $m>1$, then $B_1\subset I_1$ and $B_2=[x_2-\lambda^4 b,x_2]$.
\item[$\vdots$]
\item If $m>p-3$, then $B_{p-3}\subset I_{p-3}$ and
$B_{p-2}=\left\langle x_{p-2}, x_{p-2}+(-1)^{p+1}\lambda^p b\right\rangle
=[0,\lambda^p b]$.
\end{itemize}
Notice that $B_{p-2}$ is of the same form as $B_{-2}$. 
What precedes implies that
\begin{gather*}
\forall i\in\Lbrack -2, m\Rbrack,\ B_i=
\left\langle x_{i\bmod p}, x_{i\bmod p}+(-1)^{r+1}\lambda^{i+2}b\right\rangle,
\text{ where }i=qp+r,\ r\in\Lbrack 0, p-1\Rbrack,\\
\forall i\in\Lbrack -2, m-1\Rbrack,\ B_i\subset I_{i\bmod p},\\
B_m\supset I_{m\bmod p}.
\end{gather*}
This ends step 2 with $n_1:=n_0+m+2$ and $j:=m$.

\medskip
\textbf{Step 3:} there exists $n_2\ge n_1$ such that $f^{n_2}(A)=[0,1]$.

Let $n_1\ge 0$ and let $j\in\Lbrack 1,p-1\Rbrack$ be such that
$f^{n_1}(A)\supset I_j$ (step 2).
In the covering graph of Figure~\ref{fig:Grecouvrement}, we see that there
exists an integer $q\ge 0$ such that, for every vertex
$C$ of the graph, there exists a path of length $q$, with only
arrows of type $\to$, starting from $I_j$ and ending at $C$.
This implies that $f^q(I_j)=[0,1]$, that is, $f^{n_1+q}(A)=[0,1]$.

\medskip
We have shown that, for every nondegenerate closed interval 
$A\subset [0,1]$, there exists $n$ such that $f^n(A)=[0,1]$. 
We conclude that $f$ is topologically mixing.
\end{proof}

\section{General case}

\subsection{Square root of a map}

We first recall the definition of the so-called \emph{square root} of
an interval map.
If $f\colon [0,b]\to [0,b]$ is an interval map, the square root
of $f$ is the continuous map
$g\colon [0,3b]\to [0,3b]$ defined by
\begin{itemize}
\item $\forall x\in [0,b]$, $g(x):=f(x)+2b$,
\item $\forall x\in [2b,3b]$, $g(x):=x-2b$,
\item $g$ is linear on $[b,2b]$.
\end{itemize}
The graphs of $g$ and $g^2$ are represented in Figure~\ref{fig:type-2n-g-g2}.

\begin{figure}[hbt]
\centerline{\includegraphics{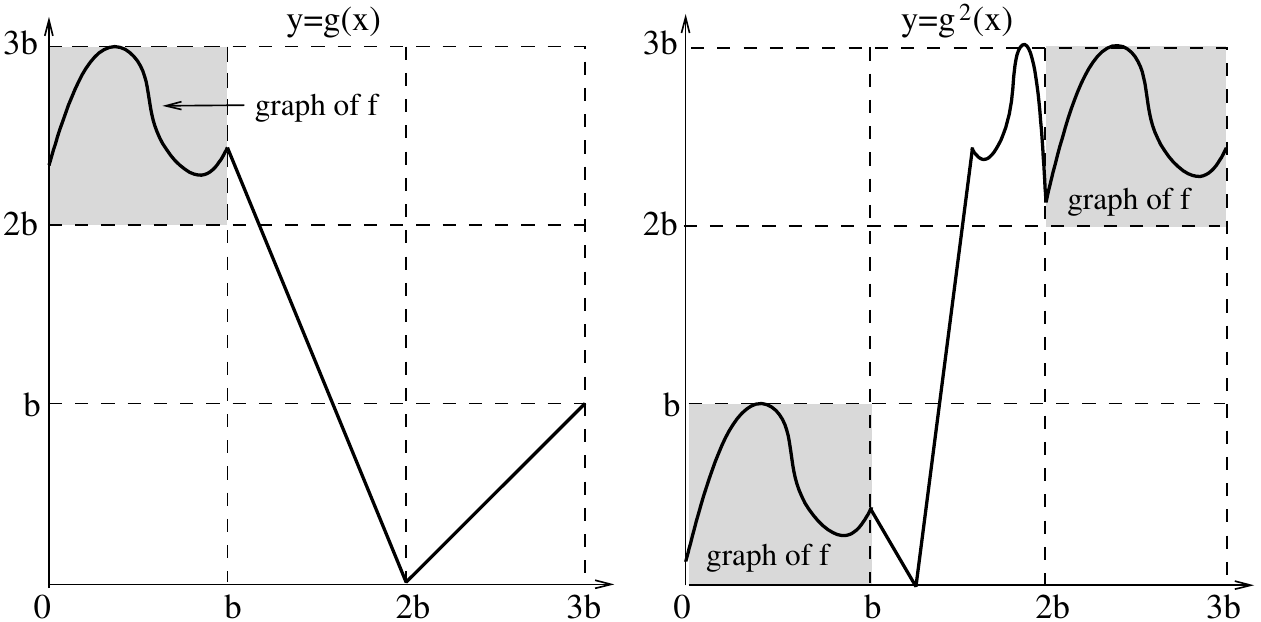}}
\caption{The left side represents the map $g$, which is the square root of $f$.
The right side represents the map $g^2$.} 
\label{fig:type-2n-g-g2}
\end{figure}

The square root map has the following properties, see e.g.
\cite[Examples 3.22 and 4.62]{R3}.

\begin{proposition}\label{prop:squareroot}
Let $f$ be an interval map of type $n$, and let $g$ be the square root of
$f$. Then $g$ is of type $2n$ and $h_{top}(g)=\frac{h_{top}(f)}2$.
If $f$ is piecewise monotone, then $g$ is piecewise monotone too.
\end{proposition}

\subsection{Piecewise monotone map of given entropy and type}

\begin{theorem}\label{theo:main}
Let $p\ge 3$ be an odd integer, let $d$ be a non negative integer and
$\lambda$ a real number such that $\lambda\ge \lambda_p$.
Then there exists a piecewise monotone map $f$ whose type is
$2^dp$ for Sharkovskii's order and such that 
$h_{top}(f)=\frac{\log\lambda}{2^d}$. If $d=0$, the map $f$ can be built
in such a way that it is topologically mixing.
\end{theorem}

\begin{proof}
If $d=0$, we take $f=f_{p,\lambda}$ defined in Section~\ref{sec:oddtype}.

If $d>0$, we start with the map $f_{p,\lambda}$, then we build the square root
of $f_{p,\lambda}$, then the square root of the square root, etc. 
According to Proposition~\ref{prop:squareroot}, after $d$ steps
we get a piecewise monotone interval map $f$ of type
$2^d p$ and such that $h_{top}(f)=\frac{h_{top}(f_{p,\lambda})}{2^d}=
\frac{\log\lambda}{2^d}$.
\end{proof}

\begin{corollary}
For every positive real number $h$, there exists a piecewise monotone
interval map $f$ such that $h_{top}(f)=h$.
\end{corollary}

\begin{proof}
Let $d\ge 0$ be an integer such that $\frac{\log \lambda_3}{2^d}\le h$
and set $\lambda:=\exp(2^dh)$. Then $\lambda\ge \lambda_3$ and,
according to Theorem~\ref{theo:main}, there exists a piecewise monotone
interval map $f$ of type $2^d3$ such that 
$h_{top}(f)=\frac{\log \lambda}{2^d}=h$.
\end{proof}

\bibliographystyle{plain}
\bibliography{../../tex/biblio/biblio} 

\begin{thebibliography}{1}

\bibitem{ALM}
Ll. Alsed{\`a}, J.~Llibre, and M.~Misiurewicz.
\newblock {\em Combinatorial dynamics and entropy in dimension one}, volume~5
  of {\em Advanced Series in Nonlinear Dynamics}.
\newblock World Scientific Publishing Co. Inc., River Edge, NJ, second edition,
  2000.

\bibitem{BGMY}
L.~Block, J.~Guckenheimer, M.~Misiurewicz, and L.~S. Young.
\newblock Periodic points and topological entropy of one dimensional maps.
\newblock In {\em Global Theory of Dynamical Systems}, Lecture Notes in
  Mathematics, no. 819, pages 18--34. Springer-Verlag, 1980.

\bibitem{MT2}
J.~Milnor and W.~Thurston.
\newblock On iterated maps of the interval.
\newblock In {\em Dynamical systems ({C}ollege {P}ark, {MD}, 1986--87)}, volume
  1342 of {\em Lecture Notes in Math.}, pages 465--563. Springer, Berlin, 1988.

\bibitem{MS2}
M.~Misiurewicz and W.~Szlenk.
\newblock Entropy of piecewise monotone mappings.
\newblock {\em Studia Math.}, 67(1):45--63, 1980.

\bibitem{R3}
S.~Ruette.
\newblock {\em Chaos on the interval}.
\newblock University Lectures series, No. 67. AMS, 2017.

\bibitem{Sha}
A.~N. Sharkovsky.
\newblock Co-existence of cycles of a continuous mapping of the line into
  itself ({R}ussian).
\newblock {\em Ukrain. Mat. \u Z.}, 16:61--71, 1964.
\newblock English translation, {\it J. Bifur. Chaos Appl. Sci. Engrg.},
  5:1263--1273, 1995.

\bibitem{Ste}
P.~{\v{S}}tefan.
\newblock A theorem of \v{S}arkovskii on the existence of periodic orbits of
  continuous endomorphisms of the real line.
\newblock {\em Comm. Math. Phys.}, 54(3):237--248, 1977.

\bibitem{You}
L.~S. Young.
\newblock On the prevalence of horseshoes.
\newblock {\em Trans. Amer. Math. Soc.}, 263(1):75--88, 1981.

\end{thebibliography}
\end{document}